\documentclass[11pt]{amsart}
\usepackage{amssymb,amsmath,amsthm,amsfonts,amsopn,url,color,hyperref,enumitem,mathrsfs}
\usepackage{csquotes}
\usepackage[all]{xy}

\theoremstyle{plain}
\newtheorem{thm}{Theorem}[section]
\newtheorem{prop}[thm]{Proposition}
\newtheorem{lemma}[thm]{Lemma}
\newtheorem{cor}[thm]{Corollary}

\theoremstyle{definition}
\newtheorem{defn}[thm]{Definition}
\newtheorem*{defn*}{Definition}
\newtheorem*{question*}{Question}

\newtheorem{example}[thm]{Example}
\newtheorem*{example*}{Example}
\newtheorem{rem}[thm]{Remark}
\newtheorem*{rem*}{Remark}
\newcommand{\field}[1]{\mathbb{#1}}
\newcommand{\N}{\field{N}}
\newcommand{\Z}{\field{Z}}

\newcommand{\ideal}[1]{\mathfrak{#1}}
\newcommand{\m}{\ideal{m}}
\newcommand{\n}{\ideal{n}}
\newcommand{\p}{\ideal{p}}

\newcommand{\func}[1]{\mathrm{#1} \,}

\newcommand{\Spec}{\func{Spec}}

\newcommand{\hgt}{\func{ht}}

\newcommand{\ra}{\rightarrow}

\DeclareMathOperator{\len}{ }

\DeclareMathOperator{\ann}{ann}

\newcommand{\be}{\begin{enumerate}}
\newcommand{\ee}{\end{enumerate}}

\newcommand{\li}
 {\leftfootline}

\renewcommand{\phi}{\varphi}

\DeclareMathOperator{\orc}{c}

\makeatletter
\def\@settitle{\begin{center}%
  \baselineskip14\p@\relax
  \bfseries
  \uppercasenonmath\@title
  \@title
  \ifx\@subtitle\@empty\else
     \\[1ex]
     \@subtitle
  \fi
  \end{center}%
}
\def\subtitle#1{\gdef\@subtitle{#1}}
\def\@subtitle{}
\makeatother

\author{Neil Epstein}
\address{Department of Mathematical Sciences \\ George Mason University \\ Fairfax, VA  22030}
\email{nepstei2@gmu.edu}

\author{Jay Shapiro}
\address{Department of Mathematical Sciences \\ George Mason University \\ Fairfax, VA  22030}
\email{jshapiro@gmu.edu}

\title{The Ohm-Rush content function III}
\subtitle{Completion, globalization, and power-content algebras}

\date{February 25, 2021}
\begin{document}
\begin{abstract}
One says that a ring homomorphism $R \rightarrow S$ is \emph{Ohm-Rush} if extension commutes with arbitrary intersection of ideals, or equivalently if for any element $f\in S$, there is a unique smallest ideal of $R$ whose extension to $S$ contains $f$, called the \emph{content} of $f$.  For Noetherian local rings, we analyze whether the completion map is Ohm-Rush.  We show that the answer is typically `yes' in dimension one, but `no' in higher dimension, and in any case it coincides with the content map having good algebraic properties. 
We then analyze the  question of when the Ohm-Rush property globalizes in faithfully flat modules and algebras over a 1-dimensional Noetherian domain, culminating both in a positive result and a counterexample.  Finally, we introduce a notion that we show is strictly between the Ohm-Rush property and the weak content algebra property.
\end{abstract}

\maketitle

\section{Introduction}
One of the most useful and important methods
 of constructing a new commutative ring from an existing one is via polynomial extension.   The polynomial algebra $ R[x]$ inherits many of the properties of $R$.   There are several important keys in the study of polynomial extensions and their relation to the base ring.  Such algebras are of course faithfully flat over $R$.  Moreover there is a natural map from elements of $R[x]$ to finitely generated ideals of $R$ via the well known content map, which sends an element $f$ to the ideal generated by the coefficients of $f$, denoted $c(f)$.

In \cite{OhmRu-content}, Ohm and Rush define a function $c$ from elements of an arbitrary $R$-algebra $S$ to the set of ideals of $R$.
For $S$ faithfully flat over $R$, they give criteria for when $c$ can serve as an appropriate generalization of the content function in polynomial extensions.
  Such extensions are called {\it content algebras}. 
  In a variation of this concept, Rush in \cite{Ru-content} examined what he called {\it weak content} algebras, whose properties are  easier to check then those of a content algebra.  Also essentially at the same time as Ohm and Rush, Eakin and Silver \cite{EakSi-almost} defined and examined some of the same properties as in \cite{OhmRu-content} and applied the results to locally polynomial rings.
More recently in a series of three papers \cite{nmeSh-OR, nmeSh-Gauss, nmeSh-OR2}, the authors further examined content and weak content algebras, as well as defined an intermediary notion (\emph{semicontent} algebras).  Nasehpour in \cite{Nas-ABconj} studied   content algebras that satisfied the additional property that $c(fg) = c(f)c(g)$, calling them \emph{Gaussian}.  That is, a Gaussian algebra is an Ohm-Rush algebra in which the content function is a homomorphism of multiplicative semigroups.

We introduce the basic definitions and terminologies in Section \ref{sec:basics} that will be used throughout.  We then give necessary and sufficient conditions (Theorem~\ref{thm:ORcompletion}) for the completion  of a local Noetherian ring $R$ to be content over $R$ (Gaussian even) in terms of the extension of ideals.   Combining this with the work of Hassler and Wiegand \cite{HaWi-extend}  on the extension of modules, we prove the main result of this section (Theorem~\ref{thm:dim1completion}) which characterizes when $\hat{R}$ is Gaussian over a 1-dimensional reduced local Noetherian ring $R$.   From this it follows (Corollary~\ref{cor:1ddomcompletion}) that if $R$ is an analytically irreducible Noetherian local integral domain of dimension one, $\hat{R}$ is Gaussian over $R$.   Moreover  Example~\ref{ex:dim2} shows that one cannot expect to extend this result to higher dimensions.

In Section \S\ref{sec:one} we characterize when an $R$-algebra is a content algebra, where $R$ is a Dedekind domain.  We also show that in many typical cases over a 1-dimensional base, the property of being an Ohm-Rush algebra  globalizes (Theorems~\ref{thm:DedOR} and ~\ref{thm:dim1ORglobal}).  However, the Ohm-Rush property does \emph{not} globalize in general, even for a faithfully flat (albeit non-Noetherian) algebra over $\mathbb{Z}$.  As far as we know, we give here the first known counterexample to globalization of the Ohm-Rush property (see Example~\ref{ex:X/p}).  We then apply our results in Example~\ref{ex:notfg} to show that a known locally polynomial algebra is in fact Gaussian.

In the final section \S\ref{sec:pc} we define and examine power-content algebras, a property that  lies strictly between Ohm-Rush algebras and weak content algebras.

\section{When is the $\m$-adic completion Ohm-Rush?}\label{sec:basics} 
In this section we examine conditions on a local Noetherian ring $R$ so that its $\m$-adic completion is a content (in fact Gaussian) algebra over $R$, but first
we begin with some of  the basic definitions that will be used throughout the paper.

\begin{defn}[See \cite{OhmRu-content}; current nomenclature from \cite{nmeSh-OR}]\label{def:orc}
Let $R$ be a ring, $M$ an $R$-module, and $f\in M$.  Then the (\emph{Ohm-Rush$)$ content} of $f$ is given by  \[
\orc(f) := \bigcap\{I \subseteq R \text{ ideal } \mid f \in IM\}.\footnote{In \cite{nmeSh-OR}, we use the symbol $\Omega$ for this function.}
\]
If $f\in \orc(f)M$ for all $f\in M$, we say that $M$ is an \emph{Ohm-Rush module}; if $M$ is moreover an $R$-algebra, we say that it is an \emph{Ohm-Rush algebra} over $R$.

When $f\in S$, we introduce the notation $L_f := \{I \subseteq R \text{ ideal } \mid f \in IM\}$.
\end{defn}

\begin{defn}\label{defs}
Let $R \rightarrow S$ be an Ohm-Rush algebra.  We say that it is  \begin{enumerate}
\item\label{def:content} a \emph{content algebra} \cite{OhmRu-content} if it is faithfully flat and for any $f, g \in S$, there is some $n\in \N$ with $\orc(f)^n \orc(g) = \orc(f)^{n-1} \orc(fg)$,
\item a \emph{Gaussian algebra} \cite{Nas-ABconj} if it is faithfully flat and for any $f, g \in S$, we have $\orc(fg) = \orc(f) \orc(g)$.  That is, one may choose $n=1$ in (\ref{def:content}).
\end{enumerate}
\end{defn}

We recall some basic facts about the above properties. In an Ohm-Rush algebra $c(fg) \subseteq c(f)c(g)$ for all $f,g \in S$ \cite[Proposition 1.1(i)]{Ru-content}, while an Ohm-Rush algebra is flat if and only if $c(af)= ac(f)$ for all $a\in R$ and $f\in S$ \cite[Corollary 1.6]{OhmRu-content}.  It is well known that if an $R$-module is flat then the extension of ideals distributes over \emph{finite} intersection.  On the other hand, an $R$-module is Ohm-Rush if and only if the extension of ideals distributes over \emph{arbitrary} intersections \cite[2.2]{EakSi-almost}. 
Also in a content algebra, prime ideals extend to prime ideals \cite[Theorem 1.2]{Ru-content}.  (In fact by that same result we know that prime ideals extend to prime ideals in a faithfully flat weak content algebra - see \S\ref{sec:pc} for a definition.)

Common examples of content algebras include polynomial algebras \cite{Ded-DM, Mer-DM}, monoid algebras where the monoid is torsion-free and cancellative \cite{No-content}, and power series algebras when the base ring is Noetherian \cite{nmeSh-DMpower}.  On the other hand, the Gaussian algebra property is much stronger, and for instance only holds for polynomial algebras over an integral domain when the domain is Pr\"ufer \cite{Gil-hilf, Pr-DM, Ts-Gauss}.  See also \cite{nmeSh-Gauss} for properties of ring elements $f\in S$ for which $\orc(f)\orc(g)=\orc(fg)$ for all $g\in S$.

Recall that a ring extension $R \subseteq S$ is called \emph{cyclically pure} if every ideal of $R$ is contracted from $S$.  That is, for every ideal $I$ of $R$, we have $IS \cap R = I$.  In particular, any faithfully flat extension is cyclically pure.

\begin{prop}\label{pr:pureprincipal}
Let $R \ra S$ be a cyclically pure (e.g. faithfully flat) ring homomorphism. Let $f\in S$.  Suppose there is some $r\in R$ such that $fS=rS$.  Then $c(f) = rR$, so $f \in c(f)S$.

If every $f\in S$ has this property, then $S$ is a Gaussian $R$-algebra. 
\end{prop}

\begin{proof}
Let $f\in S$ and $r\in R$ with $fS=rS$. Let $I$ be an ideal of $R$ such that $f\in IS$.  Then $r\in fS \cap R \subseteq IS \cap R = I$ by cyclic purity.  Since this holds for all such ideals $I$, we have $r \in c(f)$.  On the other hand, since $f \in rS$, we have $c(f) \subseteq rR$.  Hence $c(f)=rR$. Thus, $f\in rS = c(f)S$.

Now suppose that \emph{every} principal ideal of $S$ is extended from a principal ideal of $R$. Let $f,g \in S$.  There exist $a,b\in R$ such that $fS=aS$ and  $gS=bS$.  Moreover, we  have $fgS = abS$.  So by the above, $c(f)c(g) = (aR)(bR) = abR = c(fg)$.  

Since we have shown that the map is Ohm-Rush, we can check flatness by showing that $\orc(af) = a\orc(f)$ for all $a\in R$, $f\in S$, due to \cite[Corollary 1.6]{OhmRu-content}. So let $f\in S$ and $a,r\in R$ with $fS=rS$.  Then $afS = arS$, so $c(af) = arR = ac(f)$.  Finally, for any maximal ideal $\m$ of $R$, we have $\m S \cap R = \m$, whence $\m S \neq S$, finishing the proof of faithful flatness and the Gaussian property.
\end{proof}

In the case of the completion map, we have the following characterization:

\begin{thm}\label{thm:ORcompletion}
Let $(R,\m)$ be a Noetherian local ring, and let $(\hat{R},\n)$  be its $\m$-adic completion.  Then the following are  equivalent: \begin{enumerate}
\item\label{it:comOR} The map $R \ra \hat{R}$ is Ohm-Rush.
\item\label{it:comGauss} The map $R  \ra \hat{R}$ is Gaussian.
\item\label{it:comprin} For any $g\in \hat{R}$, there is some $r\in R$ such that $g\hat{R} = r\hat{R}$.  That is, every principal ideal of $\hat{R}$ is extended from a principal ideal of $R$.
\item\label{it:comideal} Every ideal of $\hat{R}$ is extended from $R$.
\end{enumerate}
\end{thm}

\begin{proof}
We have (\ref{it:comGauss})   $\implies$  (\ref{it:comOR})   by definition, and Proposition~\ref{pr:pureprincipal} shows that (\ref{it:comprin}) $\implies$ (\ref{it:comGauss}).  To see that (\ref{it:comideal}) $\implies$ (\ref{it:comprin}), let $g\in \hat{R}$.  Then there is some ideal $I$ of $R$ with $g\hat{R} = I\hat{R}$, which by \cite[Remark 3.3]{nmeSh-Gauss} must be a principal ideal of $R$.

It remains only to show that (\ref{it:comOR}) $\implies$ (\ref{it:comideal}).  So suppose the completion map is Ohm-Rush.  Let $0 \neq g\in \hat{R}$.  For each $t\in \N$, choose $g_t \in R$  such that $g-g_t \in \n^t$.  Set $I_t := g_tR + \m^t$.  Then $I_t \hat{R} = g\hat{R} + \n^t$.  By the Krull intersection theorem applied to the quotient ring $\hat{R}/g\hat{R}$, then, we have $\bigcap_t (I_t\hat{R})=g\hat{R}$. Thus,  $c(g) \subseteq \bigcap_t I_t = (\bigcap_t I_t)\hat{R} \cap R  = (\bigcap_t I_t \hat{R}) \cap   R   = g\hat{R} \cap R$.  Hence, $g\hat{R}  \subseteq c(g)\hat{R} \subseteq (g\hat{R} \cap R)\hat{R} \subseteq g\hat{R}$.  It   follows that $g\hat{R}=c(g)\hat{R}$.  Now let $J$ be any nonzero ideal of $\hat{R}$.  Say $J= (f_1, \ldots, f_n)$ with each $f_i \neq 0$.  Then by the above, $J = (\sum_{i=1}^n c(f_i)) \hat{R}$.
\end{proof}

This yields the following useful necessary criterion for Ohm-Rushness of the completion map:

\begin{cor}
Let $(R,\m)$ be a Noetherian local ring and $(\hat{R}, \n)$ its $\m$-adic completion.  Suppose there is some nonzero $g\in \hat{R}$ such that $g\hat{R} \cap R = 0$.  Then the completion map $R \ra \hat{R}$ is not Ohm-Rush.
\end{cor}

\begin{proof}
If it is Ohm-Rush, then by  Theorem~\ref{thm:ORcompletion}, $g\hat{R}= r\hat{R}$ for some $r\in R$. In particular, $r\neq 0$.  But then $0=g\hat{R} \cap R =   (rR)\hat{R} \cap R = rR \neq 0$, which is absurd.
\end{proof}

This theorem allows us to use a result of Hassler and Wiegand regarding extended \emph{modules}.  First recall the following:

\begin{defn}\cite{HaWi-extend}
Let $R$ and $S$ be Noetherian local rings and $(R,\m) \ra (S,\n)$
 a flat local homomorphism. Given a finitely generated $S$-module $N$, we say $N$ is \emph{extended} (from $R$) provided there is an $R$-module $M$ such that $S \otimes_RM$ is isomorphic to $N$ as an $S$-module.
\end{defn}

\begin{rem}\label{rem:extproperties}
Note that the $R$-module $M$ is forced to be finitely generated, by \cite[Ch. 1, \S3(6), Proposition 11]{Bour-CA}.  Moreover, if $\m S = \n$, then $\mu_R(M) = \mu_S(N)$ (where we use the symbol $\mu$ to denote the minimal number of generators of a module).  To see this, recall the well-known formula that for a finite-length $R$-module $L$, if $S/\m S$ has finite vector-space dimension over $R$ we have $\len_S(S \otimes_R L) = \len_S(S/\m S)\cdot \len_R(L)$.  Thus, \begin{align*}
\mu_S(S \otimes_R M) &= \len_S((S \otimes_R M) / \n(S \otimes_R M)) = \len_S((S \otimes_R M) / \m (S \otimes_R M)) \\
&= \len_S (S \otimes_R M / \m M) = \len_S(S/\m S) \cdot \len_R(M/\m M) = 1 \cdot \mu_R(M).
\end{align*}
In particular, $N$ is a cyclic $S$-module iff $M$ is a cyclic $R$-module.

It follows that $J$ is an extended ideal from $R$ (in the usual sense) $\iff S/J$ is an extended $S$-\emph{module} from $R$ in the Hassler-Wiegand sense.  For one direction, if $J$ is an extended ideal, $J=IS$ for some ideal $I$ of $R$, whence $S/J = S/IS \cong S \otimes_R R/I$.  For the other direction, if $S/J$ is extended from $R$, say $S/J \cong S \otimes_R M$, then by the above, we have $1=\mu_S(S/J) = \mu_R(M)$, whence $M$ is cyclic. In particular, $M\cong R/I$, where $I=\ann_R(M)$.  Thus, $S/J \cong S \otimes_R R/I = S/IS$ as $S$-modules, whence $J=IS$.
\end{rem}

Recall the following useful result on extended modules:

\begin{prop}[{\cite[Corollary 4.5]{HaWi-extend}}]\label{pr:HW-extend}
Let $(R,\m)$ and $(S, \n)$ be one-dimensional Noetherian local rings, and let $(R,\m) \ra (S, \n)$ be a flat local homomorphism such that $\n=\m S$ and the induced map $R/\m \ra S/\m S$ of residue fields is an isomorphism. Let $K(S)$ be the quotient ring of $S$ obtained by inverting the complement of the union of the height zero primes of $S$.  The following are equivalent: \begin{enumerate}
    \item For any finitely generated $S$-module $N$ such that $K(S) \otimes_S N$ is projective as a $K(S)$-module, $N$ is extended from $R$.
    \item The natural map $\Spec(S) \ra \Spec(R)$ is bijective.
\end{enumerate}
\end{prop}

Next is the main theorem of this section.

\begin{thm}\label{thm:dim1completion}
Let $(R,\m)$ be a $1$-dimensional reduced Noetherian local ring.  Then the following are equivalent: \begin{enumerate}
    \item\label{it:allextend} Every finitely generated $\hat{R}$-module is extended from $R$.
    \item\label{it:idealextend} Every ideal of $\hat{R}$ is extended from $R$.
    \item\label{it:prextend} Every prime ideal of $\hat{R}$ is extended from $R$.
    \item\label{it:redbij} $\hat{R}$ is reduced and the natural map $\Spec(\hat{R}) \rightarrow \Spec(R)$ is a bijection.
    \item\label{it:comORdim1} The completion map $R \rightarrow \hat{R}$ is Ohm-Rush.
    \item\label{it:comGaussdim1} The completion map $R \ra \hat{R}$ is Gaussian.
\end{enumerate}
\end{thm}

\begin{proof}
If $\hat{R}$ is reduced and the Spec map is bijective, then $K(\hat{R})$ is a finite product of fields, so all modules over it are projective.  Hence by Propositions~\ref{pr:HW-extend}, all finite $\hat{R}$-modules are extended.  Thus (\ref{it:redbij}) $\implies$ (\ref{it:allextend}).  The fact that (\ref{it:allextend}) $\implies$ (\ref{it:idealextend}) follows from Remark~\ref{rem:extproperties}.  The implication (\ref{it:idealextend}) $\implies$ (\ref{it:prextend}) is trivial.  The equivalence of the three conditions (\ref{it:idealextend}), (\ref{it:comORdim1}) and (\ref{it:comGaussdim1}) follows from Theorem~\ref{thm:ORcompletion}.  

For the implication (\ref{it:prextend}) $\implies$ (\ref{it:redbij}), suppose that all prime ideals of $\hat{R}$ are extended from $R$. If $P$ is a prime ideal of $\hat{R}$, let $J$ be an ideal of $R$ with $P = J\hat{R}$.  Then $P \cap R = J\hat{R} \cap R = J$ by purity of the map $R \ra \hat{R}$.  Hence, $P = J\hat{R} = (P \cap R)\hat{R}$.  Now let $P, P' \in \Spec \hat{R}$ with $P \cap R = P' \cap \hat{R}$.  Then $P = (P \cap R)\hat{R} = (P' \cap R)\hat{R}= P'$. Hence, the Spec map is injective.  For surjectivity, we need to show that all prime ideals of $R$ are contracted from prime ideals of $\hat{R}$, so let $\p$ be a prime ideal of $R$.  If $\p=\m$, then $\p=\m \hat{R} \cap R = \m$ is contracted from the maximal ideal of $\hat{R}$.  If $\p$ is a minimal prime of $R$, then by the Going-Down property (associated to the inclusion $\p \subset \m$ and the contraction $\m=\hat{\m} \cap R$) there is some prime $P$ of $\hat{R}$ lying over $\p$.  Since prime ideals of $\hat{R}$ are extended, we have $P=(P \cap R)\hat{R} = \p \hat{R}$.  Thus, $\p \hat{R}$ is prime and $\p \hat{R} \cap R = \p$, so the Spec map is surjective.  Thus, it is bijective, and we have moreover shown that every prime ideal of $\hat{R}$ is of the form $\p\hat{R}$, where $\p$ is a prime ideal of $R$.

Finally, we want to prove that $\hat{R}$ is reduced.  For this, let $P_1, \ldots, P_n$ be the minimal primes of $\hat{R}$.  Then we have $P_j = \p_j \hat{R}$, where $\p_1,\ldots, \p_n$ are the minimal primes of $R$.  So the nilradical of $\hat{R}$ is $\bigcap_{i=1}^n P_i = \bigcap_i (\p_i \hat{R}) = (\bigcap_i \p_i) \hat{R}$ since the completion map is flat. But $\bigcap_{i=1}^n \p_i$  is the nilradical of $R$, hence $0$ since $R$ is reduced.  Thus, the nilradical of $\hat{R}$ is also zero.
\end{proof}

In particular, for a 1-dimensional Noetherian local ring whose completion is a domain, the completion map is Gaussian:

\begin{cor}\label{cor:1ddomcompletion}
Let $(R,\m)$ be a 
 Noetherian local integral domain of Krull dimension 1.  Then the map $R \ra \hat{R}$ is Gaussian if and only if it is Ohm-Rush if and only if $R$ is analytically irreducible.
\end{cor}

However, the above results do not extend to higher dimension, even in the case of regular local rings essentially of finite type over a field, as the following example demonstrates:

\begin{example}\label{ex:dim2}
Let $R=k[x,y, z_1, \ldots, z_n]_{(x,y, z_1,\ldots, z_n)}$ where  $k$ is a field with char$(k)\neq 2$.  Then $\hat{R}=k[\![x,y, z_1,...z_n]\!]$.  Let $f := x^2-y^2(y+1)$.  By the Eisenstein irreducibility criterion applied to the prime element  $y+1$ of $k[y]$ in the polynomial extension $R=k[y][x, z_1,...,z_n]$, we have that $f$ is a prime element of $R$.   However, the coefficients in the Maclaurin expansion of the square root of $y+1$ can be written with denominators that only involve powers of $2$.  Hence, since $1/2 \in R \subseteq \hat{R}$, $y+1$ has a square root in $\hat{R}$.  Therefore, $f=x^2 - y^2(y+1) = (x+y\sqrt{y+1})(x-y\sqrt{y+1})$ presents $f \in \hat{R}$ as a nontrivial product of elements of the maximal ideal of $\hat{R}$.  It follows that $f\hat{R}$ is neither prime nor the unit ideal.
Hence, since the extension $f\hat{R}$ of the prime ideal $fR$ of $R$ is not prime in $\hat{R}$, $R \ra \hat{R}$ is not a content algebra.  But then by Theorem~\ref{thm:ORcompletion}, it cannot even be Ohm-Rush.

We also note that if char$(k)=2$, then $f :=x^3- y^3(y+1)$ is also irreducible  in $R$  by the Eisenstein criterion.   We leave it to the reader to apply (the generalized) Hensel's Lemma to see that $y+1$ has a cube root in $k[\![y]\!]$, from which it follows that $f$ is not irreducible in $\hat{R}$.  Thus in this case as well $\hat{R}$ is not a content $R$-algebra.
\end{example}

\section{Globalization of the Ohm-Rush property over a one-dimensional Noetherian domain}\label{sec:one}

In this section, we analyze the question of whether and when the Ohm-Rush property globalizes over a 1-dimensional Noetherian domain base.  That is, when $N$ is a faithfully flat $R$-module such that $N_\m$ is an Ohm-Rush $R_\m$-module for every maximal ideal $\m$ of $R$, does it follow that $N$ is an Ohm-Rush $R$-module?  We will see that the answer is `yes' if \begin{enumerate}
    \item $R$ is a Dedekind domain and $N$ is a Noetherian module over some $R$-algebra (See Theorem~\ref{thm:DedORglobal}), or 
    \item $R$ is a 1-dimensional integral domain and $N$ is a Noetherian $R$-algebra that is an integral domain (See Theorem~\ref{thm:dim1ORglobal}).
\end{enumerate}
We end the section with two examples, one of which is a faithfully flat $\Z$-algebra that is Ohm-Rush locally but not globally.

We start with a criterion to detect when a faithfully flat algebra over a DVR is a content algebra.

\begin{prop}\label{pr:DVRbase}
Let $(R,\m)$ be a DVR, and let $S$ be a faithfully flat $R$-algebra. Then $S$ is a content $R$-algebra (equivalently Gaussian) if and only if $\m S \in \Spec S$ and $\bigcap_n \m^n S = 0$.
\end{prop}

\begin{proof}
 By \cite[Exercise 1.1(5)]{Kap-CR}, $J = \bigcap_n \m^nS$ is a prime ideal of $S$. If $R \ra S$ is a content algebra then we know that $\m S$ and $(0)$ are prime ideals of $S$, and $\hgt(\m S) = \hgt(\m) = 1$ by \cite[Theorem 5.4]{nmeSh-OR2}. So the intersection must be $0$.

For the converse, first note by \cite[Proposition 2.1]{OhmRu-content} that $S$ is an Ohm-Rush $R$-algebra. 
Then since $\m S$ is prime, and $\m$ is the only maximal ideal of $R$, an appeal to \cite[Theorem 4.7]{nmeSh-OR} finishes the proof.
\end{proof}

For Dedekind domain bases, the criteria for a faithfully flat map being Ohm-Rush are a bit more subtle.  We begin with a lemma about detecting the Ohm-Rush property over a 1-dimensional Noetherian domain:

\begin{lemma}\label{lem:dim1content}
Let $R$ be a 1-dimensional Noetherian domain.  Let $M$ be a flat $R$-module $($or at least an $R$-module where extension of ideals commutes with \emph{finite} intersection$)$, and let $0\neq f\in M$.  Then $f\in c(f)M$ if and only if $L_f$ satisfies the descending chain condition.
\end{lemma}

\begin{proof}
If $f\in c(f)M$, then $c(f)\neq 0$, and $c(f)$ is the unique minimal element of $L_f$.  Hence, the ideals of $L_f$ are in one-to-one order-preserving correspondence with the ideals of the Artinian ring $R/c(f)$.  Hence $L_f$ satisfies DCC.

Conversely, suppose $L_f$ satisfies DCC.  Since it is nonempty (e.g. $R\in L_f$), it contains a minimal element.  Moreover, suppose that $I, I' \in L_f$, with $I$ minimal.  Then we have $f \in IM \cap I'M = (I \cap I')M$, so $I \cap I' \in L_f$.  By minimality of $I$, it follows that $I = I \cap I'$, whence $I \subseteq I'$.  Thus, every element of $L_f$ contains $I$, so $I=c(f) \in L_f$, whence $f \in c(f)M$.
\end{proof}

We next present a criterion for a flat module over a Dedekind domain to be Ohm-Rush.

\begin{thm}\label{thm:DedOR}
Let $R$ be a Dedekind domain, and let $M$ be a flat (i.e. torsion-free) $R$-module.  Let $0\neq f \in M$.  Then $f\in c(f)M$ if and only if the following two conditions hold: \begin{enumerate}
\item\label{it:finprimes} $L_f \cap \Spec R$ is finite, and
\item\label{it:finorder} There is some $n\in \N$ such that for all $\p \in \Spec R$, $f \notin \p^n M$.
\end{enumerate}
Hence $M$ is an Ohm-Rush module if and only if (\ref{it:finprimes}) and (\ref{it:finorder}) hold for every nonzero $f\in M$.
\end{thm}

\begin{proof}
First suppose $f \in c(f)M$.  Since $R$ is a Dedekind domain and $c(f) \neq 0$, there is a unique (up to ordering) prime decomposition of $c(f)$.  Say $c(f) = \prod_{j=1}^t P_j^{n_j}$, where the $P_j$ are distinct maximal ideals of $R$, $t\geq 0$, and each $n_j \geq 1$.  If $P \in L_f \cap \Spec R$, then $f \in PM$, whence $c(f) \subseteq P$.  Thus some $P_j \subseteq P$, so $P_j=P$.  That is, $L_f \cap \Spec R = \{P_j \mid 1\leq j \leq t\}$ is finite.  Now set $n := \max\{n_j \mid 1\leq j \leq t\}+1$.  If $f\in \p^n M$ for some $\p$, then $\p = P_j$ for some $1\leq j \leq t$.  Then $f\in P_j^n M$, whence $c(f) \subseteq P_j^n$.  It follows that   $P_j^{n_j} \subseteq P_j^n$, whence $n_j \geq n$, which is a contradiction.

Conversely, suppose $f\notin c(f)M$.  Then by Lemma~\ref{lem:dim1content}, $L_f$ admits an infinite descending chain $J_0 \supsetneq J_1 \supsetneq J_2 \supsetneq \cdots$.  Suppose that (\ref{it:finprimes}) holds.  Say $L_f \cap \Spec R = \{P_1, \ldots, P_t\}$. Then these are the only primes in the prime decompositions of the $J_i$.  By the proper containment condition, for each $j \in \N$, there is some $a_j \in \{1, \ldots, t\}$ such that $J_{j+1} \subseteq J_j P_{a_j}$.  Let $n\in \N$.  By the pigeonhole principle, $J_{(n-1)t + 1} \in P_i^n$ for some $1\leq i \leq t$.  Since $n$ was arbitrary, (\ref{it:finorder}) fails.
\end{proof}

The above then allows us to show that the property of being a faithfully flat Ohm-Rush algebra over a Dedekind domain globalizes, at least when the target ring is Noetherian.  We will later see (cf. Example~\ref{ex:X/p}) a counterexample over $\Z$ when the target ring is not Noetherian.

\begin{thm}\label{thm:DedORglobal}
Let $R$ be a Dedekind domain. Let $R \ra S$ be a ring homomorphism.  Let $M$ be a Noetherian $S$-module, considered as an $R$-module via restriction of scalars.  Suppose that for each maximal ideal $\m$ of $R$, the $R_\m$-module $M_\m$ is faithfully flat and Ohm-Rush.  Then $M$ is a faithfully flat Ohm-Rush $R$-module.
\end{thm}

\begin{proof}
Since faithful flatness globalizes, we have that $M$ is faithfully flat (and hence also torsion-free) over $R$.  Let $0\neq f\in M$.  We want to show that it satisfies conditions (\ref{it:finprimes}) and (\ref{it:finorder}) of Theorem~\ref{thm:DedOR}.

First note that since $M$ is torsion-free over $R$, for any maximal ideal $\m$ of $R$ we have $f/1 \neq 0$ in $M_\m$.  Then by Theorem~\ref{thm:DedOR} (applied to the flat Ohm-Rush $R_\m$-module $M_\m$), there is some positive integer $n$ (dependent on $\m$) with $f/1 \notin \m^n M_\m$.  It follows that $f\notin \m^n M$.

Next we prove (\ref{it:finprimes}).
Accordingly, let $0\neq f\in M$.  Suppose $f \in \p M$ for infinitely many $\p \in \Spec R$.  Enumerate a countable set of such $\p_i$, $i\in \N$, and for each $i$ find the unique $n_i$ such that $f \in \p_i^{n_i} M \setminus \p_i^{n_i + 1}M$, whose existence is guaranteed by the previous paragraph.  Let $Q_0 := Sf$, and inductively for each $i \geq 1$, set $Q_i := (Q_{i-1} :_M \p_i^{n_i})$.   This is an ascending chain of $S$-submodules of $M$  (i.e. $Q_{i-1} \subseteq Q_i$ for all $i\geq 1$).  Since $M$ is a Noetherian $S$-module, we just need to prove the chain is strict to get a contradiction.

\noindent \textbf{Claim 1}: For all $t\geq 1$, $Q_{t-1} \subseteq \p_t^{n_t} M$.

\begin{proof}[Proof of Claim 1] Let $x\in Q_{t-1}$.  Then by an easy induction, we have \[
\left(\prod_{i=1}^{t-1} \p_i^{n_i}\right) x \subseteq Q_0 = Sf \subseteq \left(\prod_{i=1}^t \p_i^{n_i}\right) M.
\] 
 To see the last containment, we have 
 \begin{align*}
 f \in \bigcap_{i=1}^t (\p_i^{n_i} M) &=  \left(\bigcap_{i=1}^t \p_i^{n_i}\right) M &\text{by flatness}\\
 &= \left(\prod_{i=1}^t \p_i^{n_i}\right) M &\text{since $R$ is a Dedekind domain}.
 \end{align*}
 Thus, we have \begin{align*}
x &\in \left(\left(\prod_{i=1}^t \p_i^{n_i}\right) M :_M \left(\prod_{i=1}^{t-1} \p_i^{n_i}\right) \right)& & \\
&=  \left(\prod_{i=1}^t \p_i^{n_i} :_R \prod_{i=1}^{t-1} \p_i^{n_i}\right)M &\text{again by flatness}&\\
&=  \p_t^{n_t} M &\text{again since $R$ is Dedekind}.& \qedhere
\end{align*}
\end{proof}

\noindent \textbf{Claim 2}: For any $t\geq 0$, we have $Q_t \nsubseteq \p_t M$.

\begin{proof}[Proof of Claim 2]
Let $W$ be the complement of the set $\bigcup_{i=1}^t \p_i$ in $R$.  Write $A = W^{-1}R$, $B = W^{-1}S$, and $N = W^{-1}M$. It is enough to show that $W^{-1}Q_t$ is not contained in $\p_t N$. Since localization commutes with colon and finite products, we may assume (by replacing $R$ with $A$) that $R$ is a \emph{semilocal} Dedekind domain, hence a PID.  After which, we have that each $\p_i$ is principal, say with generator $p_i$.  Thus,  $f=p_1^{n_1} \cdots p_t^{n_t}g$ for some $g\in M$.  Therefore, $Q_t = (Sf :_M (\prod_{i=1}^t p_i^{n_i}))$, whence $g\in Q_t$. 
So if the claim is false, $g \in p_tM$, which implies that $f \in p_t^{n_t+1}M$, a contradiction to our construction of $n_t$. This finishes the proof of the claim.
\end{proof}

Since $Q_{t-1} \subsetneq Q_t$ for all $t\geq 1$, we have a strict ascending chain of $S$-submodules of $M$, contradicting the fact that $M$ is Noetherian as an $S$-module.

Finally, to prove (\ref{it:finorder}), let $\p_1, \ldots, \p_k$ be the elements of $L_f \cap \Spec R$.  By the second paragraph of the proof, for each $1\leq i \leq k$ there is some positive integer $n_i$ with $f \notin \p^{n_i}M$.   Set $n := \max\{n_i : 1 \leq i \leq k\}$.  Then for each $i$, we have $f \notin \p_i^n M$.  But for any other maximal ideal $\m$ of $R$, we have $f \notin \m M$, whence $f \notin \m^n M$, completing the proof of (\ref{it:finorder}) from Theorem~\ref{thm:DedOR}, and thus the proof that $M$ is an Ohm-Rush $R$-module.
\end{proof}

\begin{cor}
Let $R$ be a Dedekind domain.  Let $R \ra S$ be a Noetherian $R$-algebra such that for every maximal ideal $\m$ of $R$, $S_\m$ is a faithfully flat Ohm-Rush $R_\m$-algebra.  Then $S$ is a faithfully flat Ohm-Rush $R$-algebra.
\end{cor}

\begin{proof}
Substitute $M=S$ in Theorem~\ref{thm:DedORglobal}.
\end{proof}

One could also obtain a corollary to Theorem~\ref{thm:DedORglobal} by substituting $R=S$, but this is unnecessary since any flat finitely generated module is projective \cite[Corollary to Theorem 7.12]{Mats}, and any projective module is Ohm-Rush \cite[Corollary 1.4]{OhmRu-content}.

By a quite different proof, we next present a theorem that weakens the condition on $R$ (to being merely a 1-dimensional Noetherian domain) but strengthens the condition on $M$ and $S$ (requiring $S$ to also be a domain and $M=S$).

\begin{thm}\label{thm:dim1ORglobal}
Let $R$ be a 1-dimensional integral domain.  Let $R \ra S$ be a faithfully flat map, where $S$ is a Noetherian integral domain.  Assume that for each maximal ideal $\m$ of $R$, we have that $R_\m \ra S_\m$ is Ohm-Rush $($where in both cases, we are inverting the multiplicative set $R \setminus \m)$.  Then $R \ra S$ is Ohm-Rush.
\end{thm}

\begin{proof}
First we prove the result under the apparently stronger assumption that for any finite set $X=\{\m_1, \ldots, \m_t\}$ of maximal ideals of $R$, we have that the map $W^{-1}R \ra W^{-1}S$ is Ohm-Rush, where $W := R \setminus \bigcup_{i=1}^t \m_i$.  We introduce the notation $R_X := W^{-1}R$ and $S_X := W^{-1}S$ for this.

Accordingly, let $0\neq f \in S$. If $c(f) = (1)$, then $f \in c(f) S$. Otherwise let $P_1,P_2, \ldots P_n$ be the primes of $S$ minimal over $fS$. Let $X := \{\p_1, \p_2, \ldots, \p_n\}$, where $\p_i = P_i\cap R$. Claim: $L_f\cap \Spec R \subseteq X$.  To see this let $\p \in L_f\cap \Spec R$.  Note that $0R \notin L_f$, since $f \neq 0$.  Hence $\hgt \p = 1$. Let $P \in \Spec S$ be minimal over $\p S$.  Then since $S$ is Noetherian and faithfully flat over $R$, $R$ is also Noetherian by \cite[Exercise 7.9]{Mats}, and so $\hgt P = 1$ by~\cite[Theorem 15.1]{Mats}.   But $fS$ is not in any height zero prime of $S$ (since, again, $f\neq 0$ and $S$ is a domain), so $P$ is minimal over $fS$.  Thus, $P=P_i$ for some $i$, whence $\p = P \cap R$ (by faithful flatness) $= P_i \cap R = \p_i \in X$.

Now let $I \in L_f$.  Then $IR_X \in L_{f/1}$ where $f/1$ is the image of $f$ in $S_X$. By our assumption $R_X \ra S_X$ is Ohm-Rush, whence $f/1 \in c(f/1)S_X$, so that in particular $c(f/1) \neq 0$ since $f/1 \neq 0$ by torsion-freeness of $S$ over $R$.  Let $J := c(f/1) \cap R$, which is thus also nonzero.  We have $c(f/1) \subseteq IR_X$ by definition.  But also we have $IR_X \cap R = I$.  To see this, first note that $V(I) \subseteq L_f \cap \Spec R \subseteq X$ by the previous paragraph.  In particular, all the associated primes of $I$ are in $X$.  So let $a\in R \setminus I$.  Then $(I : a) \subseteq \p$ for some associated prime $\p$ of $R/I$ (since $R$ is Noetherian), hence for some $\p \in X$.  Thus, $a/1 \notin IR_\p$.  Since $I R_\p \supseteq IR_X$, we conclude that $a \notin IR_X \cap R$. Thus, $J = c(f/1) \cap R \subseteq IR_X \cap R \subseteq I$.  But since $I \in L_f$ was arbitrary, we have that $L_f$ is in one-to-one correspondence with some set of ideals in the Artinian ring $R/J$.  Hence $L_f$ satisfies the descending chain condition.  Thus by Lemma~\ref{lem:dim1content}, $f \in c(f)S$.

To complete the proof, all we need is the following result.
\end{proof}

\begin{prop}
  Let $R$ be a Noetherian, 1-dimensional, semilocal domain. Let $S$ be a Noetherian $R$-algebra. If $S_\m$ is Ohm-Rush over $R_\m$ for all maximal ideals $\m$ of $R$, then $S$ is Ohm-Rush over $R$.
\end{prop}

\begin{proof}
Let $0 \neq f \in S$ with $c(f) \neq (1)$. Let $I \in L_f$. Let $\m_1,\ldots, \m_n$ be the set of maximal ideals of $R$. For each $1\leq i \leq n$, let $c_i$ be the content function associated to $R_{\m_i} \ra S_{\m_i}$.  Since $f \in IS$, we have $f/1 \in IS_{\m_i} = (IR_{\m_i})S_{\m_i}$, whence $IR_{\m_i} \in L_{f/1}$ for each $i = 1, \ldots, n$. Thus $0 \neq c_i(f/1) \subseteq IR_{\m_i}$ for each $i$. Hence $c_i(f/1) \cap R$ is not zero and is contained in $IR_{\m_i}\cap R$. Furthermore $I = \bigcap_i (IR_{\m_i} \cap R)$. Hence $\bigcap_i (c_i(f/1) \cap R) \subseteq I$. 

Now let $J=\bigcap_i (c_i(f/1) \cap R)$.  Since $J$ is a finite intersection of nonzero ideals in the integral domain $R$, we must have that $J \neq 0$.  But $J$ is contained in every element of $L_f$.  Since $R/J$ is artinian, $L_f$ thus satisfies the descending chain condition. So $f \in c(f)S$ by Lemma~\ref{lem:dim1content}.
\end{proof}

It is reasonable to ask whether the Noetherian condition on the $S$-module $M$ in Theorem~\ref{thm:DedORglobal} and on the ring $S$ in Theorem~\ref{thm:dim1ORglobal} can be dropped.  The following example shows that these conditions are necessary.

\begin{example}\label{ex:X/p}
At the beginning of \cite{EakSi-almost}, they give the following example.  Let $R=\mathbb Z$, let $x$ be an indeterminate, and let $S$ be the $\mathbb{Z}$-subalgebra of $\mathbb{Q}[x]$ generated over $\mathbb Z$ by the elements $x/p$, taken over all positive prime numbers $p$.  As they note, this ring is locally a polynomial ring over $\mathbb{Z}$, for if we take any particular prime number $p$, localizing $\mathbb{Z} \ra S$ at the multiplicative set $\mathbb{Z} \setminus p\mathbb{Z}$ yields the ring map $\mathbb{Z}_{(p)} \ra \Z_{(p)}[x/p]$.  Therefore, $S_{p\Z} \cong \Z_{(p)}[x/p]$ satisfies the much weaker condition of being an Ohm-Rush $\Z_{(p)}$-algebra.  However, $\Z\ra S$ is not Ohm-Rush, since for any prime number $p$, we have $x =p(x/p) \in pS$, whence $c(x) \subseteq \bigcap_p p\mathbb Z = 0$, even though $x\neq 0$, so that $x \notin c(x)S$.
\end{example}

\begin{example}\label{ex:notfg}
From the same article, we can recover a Gaussian algebra that is locally polynomial but not finitely generated.  Namely, in \cite[Example 3.15]{EakSi-almost}, they construct a $\mathbb Z$-algebra $S$ that is a Noetherian UFD that is not finitely generated over $\mathbb Z$, such that for every positive prime number $p$, we have $(\mathbb Z \setminus p\mathbb Z)^{-1}S \cong \mathbb{Z}_{(p)}[x]$.  Then either Theorem~\ref{thm:dim1ORglobal} or Theorem~\ref{thm:DedORglobal} is enough to show that $S$ is an Ohm-Rush $R$-algebra.  After this, an appeal to the fact that the Gaussian property globalizes \cite[Proposition 3.3]{nmeSh-OR2} proves that $S$ is a Gaussian $\mathbb{Z}$-algebra, since a polynomial extension in one variable of the Pr\"ufer domain $\mathbb{Z}_{(p)}$ is always Gaussian by Gauss's Lemma.
\end{example}

\section{Power-content algebras}\label{sec:pc}
  Rush in 
\cite{Ru-content} defined  a {\it weak content} algebra over $R$ as an Ohm-Rush algebra $S$ such that $\sqrt{c(fg)} = \sqrt{c(f)c(g)}$ for all $f,g\in S$.  As indicated by the terminology, any content algebra is a weak content algebra.  In this final section, we explore a property strictly between Ohm-Rush and weak content algebra.

First, we recall the notion of the content of an \emph{ideal}:

\begin{defn}[{\cite[just prior to Lemma 3.8]{nmeSh-OR}}]
Let $R \ra S$ be an Ohm-Rush algebra and $J$ an ideal of $S$.  Then $c(J) := \bigcap \{ I \subseteq R \text{ ideal} : J \subseteq IS\}$.  Equivalently, $c(J) = \sum_{g\in S} c(g)$.  Hence, $J \subseteq c(J)S$.
\end{defn}

\begin{defn}
Let $R\ra S$ be an Ohm-Rush algebra.  We say it is a \emph{power-content} algebra if for   any ideal $J$ of $S$, we have $c(\sqrt{J}) \subseteq \sqrt{c(J)}$.
\end{defn}

\begin{lemma}\label{lem:PCArad}
Let $R \ra S$ be an Ohm-Rush algebra.  Then it is a power-content algebra if and only if for any radical ideal $I$ of $R$, we have that $IS$ is a radical ideal of $S$.
\end{lemma}

\begin{proof}
Suppose we have a power-content algebra, and let $I$ be a radical ideal of $R$.  Let $f \in \sqrt{IS}$.  That is, there is some $n$ with $f^n \in IS$.  Then $c(f^n) \subseteq I$.  Since $f \in \sqrt{(f^n)}$, we have $c(f) \subseteq c(\sqrt{(f^n)}) \subseteq \sqrt{c(f^n)} \subseteq \sqrt{I} = I$.  Thus, $f\in IS$, whence $IS$ is radical.

Conversely, suppose radical ideals extend to radical ideals.  Let $J$ be an ideal of $S$.  Then $\sqrt{c(J)}S$ is a radical ideal of $S$, and $J \subseteq c(J)S \subseteq \sqrt{c(J)}S$, whence we have $\sqrt{J} \subseteq \sqrt{c(J)}S$.  Thus, $c(\sqrt{J}) \subseteq \sqrt{c(J)}$.
\end{proof}

We immediately see two distinctions among content-defined classes of $R$-algebras.

\begin{example}
Not all faithfully flat Ohm-Rush algebras are power-content.  For instance, consider the ring homomorphism $R \ra R[x]/(x^2)=:S$, where $R$ is any commutative ring and $x$ is an indeterminate over $R$.  Then it is faithfully flat and Ohm-Rush, since $S$ is a free $R$-module of rank 2.  But it is not power-content, since $\sqrt{c(0)} = \sqrt{0}$, but $c(\sqrt{0}) \supseteq c(x) = R$, whereas the nilradical of a ring is always a proper ideal. 
\end{example}

\begin{example}
Not all faithfully flat power-content algebras are weak content algebras.  For instance, let $R$ be any commutative ring, let $x,y$ be any indeterminates over $R$, and consider the algebra $R \ra R[x,y]/(xy) \cong S$.  Again it is a faithfully flat Ohm-Rush algebra because it is free as an $R$-module.  Let $\p$ be a prime ideal of $R$.  Then $xy=0 \in \p S$, whereas $x\notin \p S$  and $y \notin \p S$.  Hence this is not a weak content algebra.  But if $I$ is a radical ideal of $R$, then $S/IS \cong (R/I)[x,y]/(xy)$ is reduced, as is easy to show.  Hence this is a power-content algebra.
\end{example}

\begin{prop}
Let $R \ra S$ be a ring homomorphism.  Then it is a weak content algebra if and only if it is a power-content algebra such that for all ideals $I,J$ of $S$, we have $c(I) \cap c(J) \subseteq \sqrt{c(I \cap J)}$.
\end{prop}

\begin{proof}
First suppose it is a weak content algebra.  
Let $I$ be a radical ideal of $R$.  Say $I = \bigcap_{\p \in U} \p$ for some subset $U \subseteq \Spec R$.  Then $IS = (\bigcap_{\p \in U} \p)S = \bigcap_{\p \in U} (\p S)$ by the Ohm-Rush property.  But since $R \ra S$ is a weak content algebra, each $\p S$ is either prime or the unit ideal, so their intersection must be a radical ideal.  Hence by Lemma~\ref{lem:PCArad}, it is a power-content algebra.  Now let $I, J$ be ideals of $S$.  Let $\p$ be a prime ideal of $R$ that contains $c(I \cap J)$.  Then $I \cap J \subseteq \p S$, so that since $\p S$ is either the unit ideal or prime, either $I \subseteq \p S$ or $J \subseteq \p S$.  Thus, either $c(I) \subseteq \p$ or $c(J) \subseteq \p$, and in either case we have $c(I) \cap c(J) \subseteq \p$.  We have shown that any prime ideal that contains $c(I \cap J)$ contains $c(I) \cap c(J)$, whence $c(I) \cap c(J) \subseteq \sqrt{c(I \cap J)}$.

Conversely, suppose $R \ra S$ is a power-content algebra such that $c(I) \cap c(J) \subseteq \sqrt{c(I\cap J)}$ for all ideals $I,J$ of $S$.  Let $f,g\in S$.  Then \begin{align*}
c(f)c(g) &\subseteq c(f) \cap c(g) \subseteq \sqrt{c((f) \cap (g))} \subseteq \sqrt{c\left(\sqrt{(f) \cap (g)}\right)}\\
&= \sqrt{c\left(\sqrt{(fg)}\right)} \subseteq \sqrt{\sqrt{c(fg)}} = \sqrt{c(fg)}. \qedhere
\end{align*}
\end{proof}

Finally we show that the power-content property is transitive and, in the presence of the Ohm-Rush property, globalizes.  These results are analogous to our results of this type for Gaussian, weak content, and semicontent algebras \cite[Propositions 3.1-3.3]{nmeSh-OR2}.

\begin{thm}
Let $\phi: R \ra S$ be a flat Ohm-Rush algebra.  The following are equivalent: \begin{enumerate}
    \item[(a)] $R \ra S$ is a power-content algebra.
    \item[(b)] For every multiplicative subset $W$ of $R$, $W^{-1}R \ra W^{-1}S$ is a power-content algebra.
    \item[(c)] For every maximal ideal $\m$ of $R$, $R_\m \ra S_\m$ is a power-content algebra, where $S_\m$ is the localization of $S$ at the multiplicative set $\phi(R \setminus \m)$.
\end{enumerate}
\end{thm}

\begin{proof}
First we prove that (a) $\implies$ (b).  Recall \cite[Theorem 3.1]{OhmRu-content} that $W^{-1}S$ is an Ohm-Rush $W^{-1}R$-algebra.  Now let $I$ be a radical ideal of $W^{-1}R$.  Let $J$ be the contraction of $I$ to $R$.  Then $J$ is also radical, whence $JS$ is radical by Lemma~\ref{lem:PCArad}, whence $W^{-1}(JS) = I(W^{-1}S)$ is a radical ideal.  Then by Lemma~\ref{lem:PCArad} again, $W^{-1}R \ra W^{-1}S$ is a power-content algebra.

Since it is obvious that (b) $\implies$ (c), it remains only to prove that (c) $\implies$ (a). Accordingly, let $I$ be a radical ideal of $R$.  Let $f\in S$ such that $f^n \in IS$.  Let $\m$ be a maximal ideal of $R$.  By assumption, $R_\m \ra S_\m$ is a power-content algebra.  Thus, $IS_\m = (I R_\m)S_\m$ is a radical ideal of $S_\m$ by Lemma~\ref{lem:PCArad}.  Since $(f/1)^n \in IS_\m$, it follows that $f/1 \in IS_\m$.  Thus $(IS :_R f) \nsubseteq \m$.  Since $\m$ was an arbitrary maximal ideal of $R$, we have $(IS :_R f) = R$, whence $f \in IS$.  Thus, $IS$ is radical, whence by Lemma~\ref{lem:PCArad}, $R \ra S$ is power-content.
\end{proof}

\begin{thm}
Let $R \ra S$ and $S \ra T$ be power-content algebras.  Then $R \ra T$ is a power-content algebra.  That is, the property is transitive.
\end{thm}

\begin{proof}
We have that $R \ra T$ is Ohm-Rush by repeated use of \cite[1.2(ii)]{OhmRu-content}.  Now let $I$ be a radical ideal of $R$.  Then by Lemma~\ref{lem:PCArad}, $IS$ is a radical ideal of $S$, whence $IT = (IS)T$ is a radical ideal of $T$.  Hence by Lemma~\ref{lem:PCArad} again, $R \ra T$ is power-content.
\end{proof}

\section*{Acknowledgment}
We offer our thanks to the anonymous referee, whose careful reading and numerous comments improved the presentation of the paper.

\providecommand{\bysame}{\leavevmode\hbox to3em{\hrulefill}\thinspace}
\providecommand{\MR}{\relax\ifhmode\unskip\space\fi MR }
\providecommand{\MRhref}[2]{%
  \href{http://www.ams.org/mathscinet-getitem?mr=#1}{#2}
}
\providecommand{\href}[2]{#2}

\end{document}